\documentclass[11pt]{amsart} 
\usepackage{amsmath, amssymb, latexsym, tikz}
\usepackage{caption,ifthen,url,fancyhdr,cite}
\usepackage[foot]{amsaddr}     
\usepackage{libertine}     
\usepackage{longtable}    
\usepackage{upquote}
\usepackage{listings}    
\usepackage{textcomp} 
\usepackage{adjustbox}

\usepackage[a4paper,right=2.6cm, left=2.6cm, top=2.6cm, bottom=2.6cm, marginpar=1.6cm]{geometry}
\usepackage{mathtools}
\usepackage{bm} 
\usepackage{bbm} 
\usepackage{xcolor} 
\usepackage{multicol}

\newtheorem{Tma}{Theorem}[section]

\newtheorem{lemma}[Tma]{Lemma}
\newtheorem{conjecture}{Conjecture}[section]
\newtheorem{proposition}[Tma]{Proposition}
\newtheorem{corollary}[Tma]{Corollary}

\theoremstyle{definition}

\newtheorem{remark}[Tma]{Remark}

\newcommand{\Mon}{{{\text{\rm{Mon}}}}}
\newcommand{\Aut}{{{\text{\rm{Aut}}}}}

\newcommand{\sym}{{\rm Sym}}
\newcommand{\Alt}{{\rm Alt}}
\newcommand{\alt}{{\rm Alt}}

\newcounter{ithmcount}

\newenvironment{ithm}{\begin{list}{{\rm \alph{ithmcount})}}{\usecounter{ithmcount}\labelwidth18pt
      \leftmargin18
pt \topsep3pt \itemsep1pt \parsep2pt}}{\end{list}}

\title{Automorphisms of Kimura Hadamard Matrices}
\author[Barrera Acevedo]{Santiago Barrera Acevedo}
\address[Barrera Acevedo]{La Trobe University, Department of Mathematical and Physical Sciences, Australia}
\email{S.Barrera.Acevedo@latrobe.edu.au}

\author[Lee]{Melissa Lee}
\address[Lee]{Monash University, School of Mathematics, Clayton 3800 VIC, Australia}
\email{Melissa.Lee@monash.edu}

\renewcommand{\leq}{\leqslant}

\date{}
\begin{document}

\vspace*{-0.5cm}

\maketitle


\begin{abstract} 
We investigate the structure of the automorphism groups of Kimura Hadamard matrices (KHMs) constructed from dihedral groups. We identify several different types of automorphisms, and show that the automorphism group of a KHM always has a subgroup isomorphic to $D_{2k}\times Q_8$, or $C_2\times D_{2k}\times Q_8$ if it is $y$-invariant. We exhibit additional automorphisms arising from the holomorph of the dihedral group under suitable structural conditions. A comparison with known examples, including those of Kimura, Niwasaki, and matrices arising from the Shinoda--Yamada construction, reveals counterexamples to a conjecture of Ó~Catha\'{\i}n and suggests that no further automorphisms occur beyond those predicted by our framework.
\newline

\end{abstract}

\section{Introduction}

A \textit{Hadamard matrix} (HM) of order $n$ is an $n \times n$ matrix $H$ with entries from $\{-1, 1\}$, satisfying the property that $HH^\intercal = nI_n$, where $I_n$ denotes the identity matrix of order $n$. Two Hadamard matrices are \textit{equivalent} if one can be transformed into the other one through a series of row and column permutations and by changing the signs of rows and columns. The \textit{automorphism group} of a HM of order $n$ comprises all pairs of signed $n\times n$ permutation matrices that leave the matrix invariant under their natural action, see Section \ref{B&N} for details. The automorphism group of a Hadamard matrix gives a measure of the symmetry of the matrix. It captures how the rows and columns in the matrix can be permuted and sometimes negated while preserving the Hadamard property. More generally, the automorphism group provides a natural framework for applying algebraic methods to the analysis of Hadamard matrices. Marshall Hall \cite{Hall} was the first to consider the automorphism group of a Hadamard matrix. He showed that all HMs of order 12 are equivalent and that the full automorphism group of one of these Hadamard matrices is a central extension of the famous Mathieu group $M_{12}$. 
Following this discovery, various authors have investigated the automorphism groups of families of HMs or HMs whose automorphism group satisfies a given symmetry property. 
For instance, Kantor \cite{Kantor1} and de Launey and Stafford \cite{DS1,DS2} characterised the automorphism groups of \textit{Paley} HMs, while Egan and Flannery \cite{Egan} computed the full automorphism group of \textit{Sylvester} HMs. De Launey and Flannery \cite{dlf} studied central regular subgroups of the automorphism group of HMs, whose existence corresponds to the so-called \emph{cocyclic} HMs. These were further investigated by Ó Cath\'ain \cite{POCDT}, who also examined HMs with two-transitive automorphism groups \cite{Cat2trans}, extending work of Kantor \cite{Kantor2}. Barrera Acevedo et al. \cite{BAOCD} studied a class of \textit{two circulant core} (TCC) HMs whose automorphism group admits a transitive permutation representation; such matrices are called cocyclic TCC HMs.

In this work, we extend the study of automorphisms of HMs to the family of \textit{Kimura Hadamard matrices} (KHMs), as defined by Kimura \cite{Kimura}. These are HMs  of order $4m+4$ of the form given in \eqref{TCCKimura} where $A,B,C$ and $D$ are $\{\pm 1\}$-block matrices of order $m$ and $\bold{1}$ denotes the all-ones row vector of length $m$. 

\begin{equation}\label{TCCKimura}
\left[\begin{array}{rrrrrrrr}
1& 1 & 1 & 1 & \bold{1} & \bold{1} & \bold{1} & \bold{1}\\
1& 1 & -1 & -1 & \bold{1} & \bold{1} & \bold{-1} & \bold{-1}\\
1& -1 & 1 & -1 & \bold{1} & \bold{-1} & \bold{1} & \bold{-1}\\
1& -1 & -1 & 1 & \bold{-1} & \bold{1} & \bold{1} & \bold{-1}\\
\bold{1}^\intercal & \bold{1}^\intercal & \bold{1}^\intercal & \bold{-1}^\intercal & A & B& C & D\\
\bold{1}^\intercal & \bold{1}^\intercal& \bold{-1}^\intercal & \bold{1}^\intercal & -B & A & D & -C\\
\bold{1}^\intercal & \bold{-1}^\intercal& \bold{1}^\intercal & \bold{1}^\intercal & -C & -D & A & B\\
\bold{1}^\intercal & \bold{-1}^\intercal& \bold{-1}^\intercal & \bold{-1}^\intercal & D & -C & B & -A\\
\end{array}\right]
\end{equation}

Let $J_m$ be the all-ones matrix of order $m$ and let $\bold{0}$ denote the all-zeros row vector. Then, the matrix in Equation \ref{TCCKimura} is a HM if and only if

\begin{equation}\label{eq1}
\begin{aligned}
    AA^\intercal + BB^\intercal + CC^\intercal + DD^\intercal &= (4m+4)I_m - 4J_m, \\
    -AB^\intercal +BA^\intercal  + CD^\intercal- DC^\intercal &= 0, \\
    -AC^\intercal - BD^\intercal + CA^\intercal +DB^\intercal &= 0, \\
    AD^\intercal  - BC^\intercal+ CB^\intercal - DA^\intercal &=0, \\
\end{aligned}
\end{equation}
\begin{equation}\label{eq12}
\begin{aligned}
    \tau 2\bold{1} + \alpha A\bold{1} + \beta B\bold{1} + \gamma C\bold{1} + \delta D\bold{1} &= 0,
\end{aligned}
\end{equation}
where
\begin{equation*}
(\tau,\alpha,\beta,\gamma,\delta) \in 
\begin{Bmatrix}
(+,+,+,+,+), & (+,+,+,-,-), & (+,+,-,+,-), & (+,+,-,-,+), \\
(-,-,-,+,+), & (-,-,+,-,+), & (-,-,+,+,-)
\end{Bmatrix}.
\end{equation*}

Kimura's motivation for introducing KHMs arose from the classification of HMs of order 28 \cite{Kimura}. Kimura and Niwasaki \cite{Kimura,KimuraNiwasaki} constructed KHMs of order $8k+4$ for a finite number of odd values $k$ from dihedral groups $D_{2k}$. Their construction uses a binary version of the matrix in Table \ref{TCCKimura}, where the block matrices $A, B, C, D$ are images of elements in $\mathbb{Z}D_{2k}$ under the right regular representation of $\mathbb{Z}D_{2k}$; see Section \ref{B&N} for details. In the literature, these matrices are referred to as KHMs of dihedral type. In what follows, we restrict our attention to this class; thus, every mention of a KHM will refer to one of dihedral type. Further, Shinoda and Yamada discovered a family of KHMs \cite{ShinodaYamada} of order $8p+4$ where $p \equiv 1 \pmod{4}$ is an odd prime and $q = 2p - 1$ is a prime power. 

In this paper, for a Kimura Hadamard matrix $H$ of order $8k+4$, with $k$ odd, we identify several natural automorphisms that generate pairwise trivially intersecting subgroups of $\mathrm{Aut}(H)$ isomorphic to $D_{2k}$ and $Q_{8}$. We focus on $k$ odd, since the known KHMs have this property. Consequently, if $H$ has order $8k+4$ with $k>3$, then $|\mathrm{Aut}(H)| \geq 2^{4}k$. In the special case where $H$ is $y$--invariant (see Section~\ref{B&N}), we show that the automorphism group may be enlarged by one additional involution and hence $|\mathrm{Aut}(H)| \geq 2^{5}k$.  We also analyse further automorphisms that arise from the holomorph of $D_{2k}$ when certain structural conditions are satisfied. Finally, we compare the automorphisms obtained with those present in the examples of Kimura and Niwasaki, as well as in matrices arising from the Shinoda--Yamada construction for all admissible primes up to $601$. In doing so, we find counterexamples to a conjecture of Ó~Cath\'ain~\cite{POCDT}, and find no additional automorphisms beyond those accounted for by our framework.
 
The paper is organised as follows. Section \ref{B&N} introduces the necessary background, definitions, and notation, and reviews the constructions of Kimura and Niwasaki, and Shinoda and Yamada. Section \ref{AutsKHMs} presents our main results on automorphisms of KHMs, while Section \ref{Discussion} discusses how the structure of a KHM constrains its automorphisms and outlines some open questions.


\section{Background and Notation}\label{B&N}

All groups considered in this paper are finite. Let $C_n=\langle x\rangle$ denote the cyclic group of order $n$ and let $C_n^\times$ denote the group of units in $C_n$. Let $D_{2k}=\langle x,y \mid x^k=1, y^2=1, x^y=x^{-1}\rangle$ denote the dihedral group of order $2k$ and let $\mathbb{Z}D_{2k}$ be the integral group ring of $D_{2k}$. Let $\text{GL}_n(\mathbb{Z})$ be the general linear group of degree $n$ over the integers and let $\text{Mat}_n(\mathbb{Z})$ be the ring of $n\times n$ matrices with integer entries. If $G$ is a group, its \textit{centre} is defined by $Z(G)=\{g\in G: gh=hg\ \text{for all}\ h\in G\}$; an element $g\in Z(G)$ is called central. We denote by $\text{Sym}(\Omega)$ and $\Alt(\Omega)$ the symmetric and alternating groups on a finite set $\Omega$; when $\Omega$  has size $n$, we identify $\Omega=\{1,\ldots,n\}$ and  write $\sym_n$ and $\alt_n$. A subgroup $G$ of $\text{Sym}(\Omega)$ is called a \textit{permutation group} of degree $|\Omega|$. A group $G$ \textit{acts} on the set $\Omega$ if there is a homomorphism $f\colon G\to\text{Sym}(\Omega)$; the image of $\omega\in\Omega$ under $f(g)$ is usually denoted $\omega^g=\omega^{\alpha(g)}$. The action is \textit{faithful} if $\ker\alpha$ is trivial. The $G$-action on $\Omega$ is \textit{transitive} if for any $\alpha, \beta \in \Omega$ there exists $g \in G$ such that $\alpha^{g} = \beta$.  The \textit{stabiliser}  of  $\omega\in \Omega$ is the subgroup $G_\omega=\{g\in G: \omega^g=\omega\}$.
More generally, for $B\subset \Omega$, the setwise stabiliser of $B$ is the subgroup $G_B=\{g\in G : B^g = B\}$. The \textit{orbit} of $\omega\in \Omega$ is the set $\omega^G = \{\omega^g : g\in G\}$.
The \textit{right regular representation} of a group $G$ is a homomorphism $\rho:G \to \mathrm{Sym}(G)$ defined by $x^{\rho(g)} = xg$ for each $g,x \in G$. The left representation $\lambda$ of $G$ is defined analogously with $x^{\lambda(g)} = gx$.

\subsection{Primitive and imprimitive actions}
Let $G\leqslant \text{Sym}(\Omega)$ be a transitive permutation group acting on a set $\Omega$. A subset $B\subset \Omega$ is called a \textit{block} for $G$ if for all $g\in G$ either $B^g=B$ or $B\cap B^g=\emptyset$. A \textit{block system} $\mathcal{B}$ for $G$ is a partition of $\Omega$ into blocks for $G$. Note that the partitions $\{\{\Omega\}\}$ and  $\{\{\omega\}\mid \omega \in \Omega\}$ always form block systems for $G$; these are called the \textit{trivial} block systems. 
A transitive action is \textit{imprimitive} if there exists a nontrivial block system for $G$, and \textit{primitive} otherwise. If $B$ is a nontrivial block for $G$, then $G$ has block system $\mathcal{B}=\{B^g:g\in G\}$ on $\Omega$ since $G$ is transitive. The latter is a \textit{system of imprimitivity} for $G$, and clearly $G$ is transitive on the blocks. Note that every such system partitions $\Omega$ into subsets of the same size; in particular, if $\Omega$ is finite, then $|B|$ divides $|\Omega|$.

The group $\text{Mon}_n(\{\pm 1\})$ consists of all $\{\pm 1\}$-monomial matrices under matrix multiplication. Note that $\text{Mon}_n(\{\pm1\})=\text{Perm}_n\ltimes \text{Diag}_n$, where $\text{Perm}_n$ denotes the subgroup of $n\times n$ permutation matrices, and $\text{Diag}_n$ denotes the normal subgroup of $n\times n$ diagonal matrices. Thus, every element $R\in \text{Mon}_n(\{\pm1\})$ admits a unique decomposition of the form $R=P_R D_R$ with $P_R\in \text{Perm}_n$ and $D_R\in \text{Diag}_n$. The group $\text{Mon}_n^2(\{\pm 1\})$ acts on the set HMs of order $n$ via $ H^{(R,S)}=RHS^\intercal$. Hadamard matrices in the same $\text{Mon}_n^2(\{\pm 1\})$-orbit are called \textit{equivalent}. The {\it automorphism group} of $H$ is defined as the stabiliser \[\Aut(H)=\{(R,S)\in\Mon^{2}_n(\{\pm 1\})\colon RHS^\intercal=H\}\,.\]
An automorphism of $H$ is called {\it strong} if it is of the form $(R,R)$. We will see in Section \ref{AutsKHMs} that a KHM has nontrivial strong automorphisms.

\subsection{A permutation representation of $\text{Aut}(H)$}
A distinguished permutation representation of the automorphism group of a Hadamard matrix has been introduced to study the induced action of this group on the rows of the matrix, \cite[Definition 2]{Cat2trans}. We briefly recall this action. Let 
\[
\pi_1 : \text{Aut}(H) \to \text{Mon}_n(\{\pm 1\})
\] 
be defined by \(\pi_1(R,S) = R\). This projection is injective, so that 
\[
\text{Aut}(H) \cong \pi_1(\text{Aut}(H)).
\] 
Next, let 
\[
\pi : \pi_1(\text{Aut}(H)) \to \text{Perm}_n
\] 
be defined by $\pi(R) = P_R$, where $P_R$ is the permutation component of $R = P_R D_R$. Denote
\[
\mathcal{A}(H) = \pi(\pi_1(\text{Aut}(H))).
\] 
Identifying $\text{Perm}_n\equiv S_n$ we have that $\mathcal{A}(H)$ is a permutation group which encodes the (unsigned) action of $\text{Aut}(H)$ on the rows of $H$. Note that $\text{Ker}(\pi)=\text{Diag}_n$, and thus $\text{Aut}(H)$ is an extension of $\mathcal{A}(H)$ by $\text{Diag}_n$.

We conclude this section with a few definitions. A square matrix $M = (m_{i,j})_{i,j}$ of order $n$ is termed \textit{circulant} if its elements satisfy the condition $m_{i,j} = m_{0,j-i}$ for all indices $i$ and $j$, with the arithmetic performed modulo $n$. We denote by \({\rm diag}(Y_1, \dots, Y_n)\) to the block diagonal matrix with blocks \(Y_1, \dots, Y_n\) along the main diagonal, and by \({\rm bdiag}(Y_1, \dots, Y_n)\) to the block anti--diagonal matrix with blocks \(Y_1, \dots, Y_n\) along the main back--diagonal. More specifically, the blocks $Y_1, \dots, Y_n$ appear in order from the top right to bottom left of the matrix. Finally, for a set $S$ and elements $s_1,s_2 \in S$ we write $\delta_{s_1}^{s_2}$ to denote the Kronecker delta symbol. That is, $\delta_{s_1}^{s_2}=1$ if $s_1=s_2$ and $\delta_{s_1}^{s_2}=0$ otherwise.


\subsection{Kimura HMs}\label{KMSExp} From this section onward, let $x^0, x^1, \dots, x^{k-1}, y, x^1y, \dots, x^{k-1}y$ be a fixed ordering of the elements of $D_{2k}$, which by simplicity we denote $\mathcal{O}$; and let 
\[
\begin{array}{cccc}
\lambda: & D_{2k} & \to & \text{GL}_{2k}(\mathbb{Z})\\ 
 & g & \mapsto & [\delta^{g^{-1}x}_{y}]_{x,y\in \mathcal{O}}
\end{array}\quad
 \text{and} \quad
\begin{array}{cccc}
\rho: & D_{2k} & \to & \text{GL}_{2k}(\mathbb{Z})\\ 
 & g & \mapsto & [\delta^{xg}_{y}]_{x,y\in \mathcal{O}}
\end{array}
\]
be the left and right regular (matrix) representations of $D_{2k}$, respectively. 
Both representations can be extended to $\mathbb{Z}D_{2k}$ by linearity; by abuse of notation we also call these representations $\lambda:\mathbb{Z}D_{2k}\to \text{Mat}_n(\mathbb{Z})$ and $\rho:\mathbb{Z}D_{2k}\to \text{Mat}_n(\mathbb{Z})$. Since both representations $\lambda$ and $\rho$ are injective homomorphisms, we identify an element of $\mathbb{Z}D_{2k}$ with its image under $\lambda$ or $\rho$, respectively. We also replace $1 \mapsto 1$ and $0 \mapsto -1$ in the matrices to align with the notation of Kimura and Niwaski \cite{KimuraNiwasaki}.


Kimura and Niwaski \cite{Kimura, KimuraNiwasaki} found elements $a,b,c,d\in \mathbb{Z}D_{2k}$, for various values $k$, such that the associated matrices $A, B, C, D$ satisfy the conditions in Equations \eqref{eq1} and \eqref{eq12}, and thus yield KHMs of orders $4(2k)+4$. In their work, they used both the left and right regular representations of $D_{2k}$. We show that to study KHMs, it suffices to consider the right regular representation of $D_{2k}$. Observe that for every $g\in D_{2k}$ the matrices $\lambda(g)$ and $\rho(g)$ have the form:
\begin{equation}\label{circulantstructure}
\lambda(g) = \left[ \begin{array}{rr} X_1 & X_2\\ X_2 & X_1 \end{array} \right], \quad
\rho(g) = \left[ \begin{array}{rr} Y_1 & Y_2\\ Y_2^\intercal & Y_1^\intercal \end{array} \right],
\end{equation}
where $X_1, X_2, Y_1, Y_2$ are 
square matrices of size $k\times k$ with $X_1, Y_1, Y_2$ circulant and $X_2$ back-circulant.
We note that $X_1,X_2,Y_1$ and $Y_2$ have constant row sum, hence so do $\lambda(g)$ and $\rho(g)$. Now, for $k$ odd, let $P$ be the $k\times k$ matrix associated with the permutation $(2,k)(3,k-1)\dots((k+1)/2,(k+3)/2)$ and let $Q=\text{diag}(P,I_{2k})$. We have $Q\rho(g)Q^{-1}=\lambda(g)$ for all $g\in D_{2k}$. A direct calculation shows the following. 

\begin{lemma}\label{equivrep}
Let $a,b,c,d\in \mathbb{Z}D_{2k}$. Then the $\pm 1$--matrices associated to $\rho(a),\rho(b),\rho(c),\rho(d)$ satisfy the conditions in Equation \eqref{eq1} if and only if $\pm 1$-matrices associated to $\lambda(a),\lambda(b),\lambda(c),\lambda(d)$ satisfy the conditions in Equation \eqref{eq1}. Moreover, if $H_1$ and $H_2$ are the KHMs obtained from the matrices associated to $\rho(a),\rho(b),\rho(c),\rho(d)$ and $\lambda(a),\lambda(b),\lambda(c),\lambda(d)$, respectively, then they are equivalent via $(R,R)$ where $R={\rm diag}(1,1,1,1,Q,Q,Q,Q)$. 
\end{lemma}

In view of Lemma \ref{equivrep}, for the rest of this work, we only consider the right regular representation of $D_{2k}$.  Let now $|w|$ denote the number of non--zero coefficients of an element $w\subseteq \mathbb{Z}D_{2k}$; observe that $w=w_1+w_2y$ with $w_1,w_2\in \mathbb{Z}\langle x\rangle$. For a given KHM $H$ with elements $a,b,c,d \in \mathbb{Z}D_{2k}$ having $\{0,1\}$–coefficients, Kimura \cite{Kimura} translated the condition $HH^{\intercal} = n I_n$ into relations among $a,b,c,d$, yielding the following set of conditions.
\begin{lemma}[\cite{Kimura}, Proposition 4]\label{parameters}
Suppose $a,b,c,d\in \mathbb{Z}D_{2k}$ yield a Kimura Hadamard matrix $H$. Then the following equations must be satisfied:
\begin{itemize}
\item[(NC1)] $|a_1|+|a_2|=k-1$,
\item[(NC2)] $|b_1|+|b_2|=|c_1|+|c_2|=|d_1|+|d_2|=k$,
\item[(NC3)] $|a_1|^2+|a_2|^2+|b_1|^2+|b_2|^2+|c_1|^2+|c_2|^2+|d_1|^2+|d_2|^2=2k^2+1$,
\item[(NC4)] $|a_1||a_2|+|b_1||b_2|+|c_1||c_2|+|d_1||d_2|=k(k-1)$.
\end{itemize}
\end{lemma}

Conditions (NC1) and (NC2) imply that $A$ has constant row sum $-2$, and $B,C,D$ have constant row $0$. It follows that, $A,B,C,D$ satisfy Equation \eqref{eq12}. Kimura used Lemma \ref{parameters} to determine all possible parameter values 
$\lvert a_1\rvert, \lvert a_2\rvert, \dots, \lvert d_1\rvert, \lvert d_2\rvert$ for $k \le 27$.  He then performed computer searches to find KHMs satisfying these conditions. The matrices found by Kimura and Niwasaki \cite{Kimura, KimuraNiwasaki}  were also found by computer search, by imposing symmetry conditions on $a, b, c, d$. We recall the main assumptions; for a full justification see \cite[Sections 3 and 4]{KimuraNiwasaki}.  Kimura and Niwasaki assumed that: 

\begin{itemize}
\item[(C1)] $|a_1|\geqslant |a_2|$,
\item[(C2)] $|b_1|,|c_1|,|d_1|$ are odd integers, 
\item[(C3)] $|b_1|\geqslant |c_1|,|d_1|$. 
\end{itemize}
 
Additionally, they assumed the following symmetry condition; a subset $w_1+w_2y\in \mathbb{Z}D_{2k}$ with $w_1,w_2\in \mathbb{Z}\langle x\rangle$ is called \textit{$y$--invariant} if $y^{-1}w_iy = w_i$ for $i=1,2$. This yields the following definition. A KHM $H$ is called \textit{$y$--invariant} if $a,b,c,d$ are $y$--invariant; this led Kimura and Niwasaki to assume that

\begin{itemize}
\item[(C4)] $H$ is $y$--invariant. 
\end{itemize}

Kimura and Niwasaki observed that $w=w_1+w_2y\in \mathbb{Z}D_{2k}$ is $y$--invariant if and only if $w_1$ and $w_2$ commute with every element of $\mathbb{Z}D_{2k}$. In addition, if $w=w_1+w_2y\in \mathbb{Z}D_{2k}$ is $y$--invariant then $\rho(w_1)$ and $\rho(w_2)$ are symmetric matrices. Thus, if $H$ is a $y$--invariant KHM then each $W\in \{A,B,C,D\}$ is symmetric, and so are their blocks $W_1$ and $W_2$. It follows from (C4) and the structure of $W$ that the blocks $W_1$ and $W_2$ are circulant and symmetric. This implies that the matrices $A,B,C,D$ commute in pairs and satisfy the conditions in Equation \eqref{eq1}. 

Another symmetry condition considered by Kimura and Niwasaki in their searches for KHMs requires that, for $a,b,c,d \in \mathbb{Z}D_{2k}$, the following equations hold.
\begin{itemize}
\item[(C5)] $a_1=j-1-a_2$,\quad $b_1=j-b_2$,\quad $c_1=1+c_2$,\quad $d_1=1+d_2$, 
\end{itemize}

where $j$ represents the sum of all the elements in $D_{2k}$. Under these assumptions, Kimura and Niwasaki found examples of KHMs for all odd $k\leqslant 29$; except for $k=15$. They also found a KHM for $k=41$. All of these matrices are $y$--invariant, except for when $k=3,11$. For $k=3,5,13$, the found KHMs satisfy (C5). They also claimed to have found a $y$--invariant KHM satisfying (C5) for $k = 37$ but did not provide the elements $a,b,c,d$ that define such a matrix. Kimura and Niwasaki remarked at the end of their last joint work on KHMs that the case $k=15$ is yet to be investigated, but they never reported anything additional for this value. 

In the following lemma, we provide bounds for each of $|a_i|, |b_i|, |c_i|, |d_i|$ for $i = 1,2$. These bounds can be useful for computational searches of KHMs.

\begin{lemma}\label{conditions}
\label{all_ones_block}
Let $a,b,c,d \in \mathbb{Z}D_{2k}$ yield a KHM $H$ of order $8k+4$ with $k\geqslant 5$. Then $|a_i| <k-1$ and $|b_i|,|c_i|,|d_i|<k$ for each $i\in \{1,2\}$.
\end{lemma}
\begin{proof}
We appeal to Lemma \ref{parameters}. Since $|a_1|+|a_2|=k-1$;
and $|b_1|+|b_2|=|c_1|+|c_2|=|d_1|+|d_2|=k$, we find that $|a_1|^2+|a_2|^2 \geq (k-1)^2/2$, $|b_1|^2+|b_2|^2 \geq k^2/2$, and the latter bound also holds for $c$ and $d$. If $a_i=k-1$ for some $i\in \{1,2\}$, then the above bounds along with (NC3) of Lemma \ref{parameters} imply
\[
3k^2/2 \leq |b_1|^2+|b_2|^2+|c_1|^2+|c_2|^2+|d_1|^2+|d_2|^2 = k^2+2k,
\]
which shows that $k=1$ or $k=3$; this is a contradiction. A similar argument holds if any of $|b_i|,|c_i|,|d_i|$ is equal to $k$.
\end{proof}

Shinoda and Yamada \cite{ShinodaYamada} presented a construction of KHMs of order $8p+4$, where $p \equiv 1 \pmod{4}$ is an odd prime and $q = 2p - 1$ is a prime power. Their construction gives explicit elements $A,B,C,D$ satisfying the conditions in Equations \eqref{eq1} and \eqref{eq12}. We recall the main steps of this construction. For a given prime $p$ satisfying the conditions above, consider the finite fields $E=\text{GF}(p)$, $F=\text{GF}(q)$ and $K=\text{GF}(q^2)$. Let $\chi_2$ be a quadratic character of $E^\times$ and define $\chi_2(0)=0$. Let $\chi_4$ and $\chi_p$ be a biquadratic character and the $p$--th residue character of $K^\ast$, respectively. Let $\chi=\chi_4\chi_p$ and define $\chi(0)=0$. Let $\zeta$ be a primitive element of $K^\ast$ and let $S_{K/F}$ be the relative trace from $K$ to $F$. Set 

\begin{equation*}
u=\sum_{r=0}^{p-1} \chi(S_{K/F}\zeta^{4r})x^r;\quad 
v=\sum_{r=0}^{p-1} \chi(S_{K/F}\zeta^{4r+p})x^r;\quad
w =\sum_{r=0}^{p-1} \chi_2(r)x^r;\quad
G = \sum_{g\in D_{2p}}g.
\end{equation*}

Then, the $\pm 1$--versions $A,B,C,D$ associated to  $a=((1-y)v+G-(1-y))/2$, $b=((1-y)u+G)/2$, and $c=d=(1-y+(1+y)w+G)/2$ satisfy the conditions Equations \eqref{eq1} and \eqref{eq12} and yield a KHM of order $4(2p)+4$; moreover, such a matrix is $y$--invariant, see \cite[Theorem 3]{ShinodaYamada}. Observe that the construction by Shinoda and Yamada covers the previously missing case $k=37$.

We conclude this section by noting that an exhaustive computer search for $y$--invariant KHMs with $k = 3, 11$ did not yield any such matrices. In light of the aforementioned computations and construction, examples of KHMs are known for all $k\leq 41$ except for $k = 15, 31, 33, 35, 39$.


\section{Automorphisms of Kimura Hadamard matrices}\label{AutsKHMs}
In this section, we study the automorphism group of a Kimura Hadamard matrix. 
Recall that $\mathcal{O}$ denotes the fixed ordering $x^0, x^1, \dots, x^{k-1}, y, x^1y, \dots, x^{k-1}y$ of elements in  $D_{2k}$. Also, recall that $\rho$ denotes the right regular representation of $D_{2k}$, and for every $g \in D_{2k}$, we have  
\[
\rho(g) = \begin{bmatrix} X_1 & X_2 \\ X_2^\intercal & X_1^\intercal \end{bmatrix},
\]
where $X_1, X_2$ are circulant square matrices of size $k$. Now, suppose $H$ is a KHM with block matrices $A, B, C, D$. Then, for each $W \in \{A, B, C, D\}$, the matrix $W$ has the form  
\[
W = \begin{bmatrix} W_1 & W_2 \\ W_2^\intercal & W_1^\intercal \end{bmatrix},
\]
where $W_1, W_2$ are circulant matrices of size $k\times k$.  Next, recall that an element $w_1 + w_2 y \in \mathbb{Z}D_{2k}$, with $w_1, w_2 \in \mathbb{Z}\langle x \rangle$, is called $y$--invariant if $w_1$ and $w_2$ commute with every element of $\mathbb{Z}D_{2k}$. Moreover, a circulant matrix $W_i$ is symmetric if its first row $(r_0, \dots, r_{k-1})$ satisfies  
\[
r_i = r_{k-i} \quad \text{for all } i \in \{1, \dots, \lfloor \frac{k-1}{2} \rfloor\}.
\]
\subsection{General automorphisms of KHMs}
\label{general_auts}
We begin our construction of automorphisms of a KHM with those arising from the natural action of the dihedral group in the right regular representation.

\begin{proposition}\label{automorphisms1}
Let $H$ be a KHM of order $8k+4$, with block matrices $A,B,C,D$, let $i\in \{1,2\}$ and let $W \in \{A,B,C,D\}$. Then the following are automorphisms of $H$.
\begin{ithm}

\item An automorphism $\sigma_1$ of order $k$ that preserves the block matrices and represents the natural simultaneous group development of each $W_1,W_2$ over $C_k$.

\item An involution $\sigma_2$ that preserves the block matrices and simultaneously swaps $W_i$ and $W_i^\intercal$ within each $W$. 

\end{ithm}
Together, these are strong automorphisms that generate a group isomorphic to $D_{2k}$.
\end{proposition}
\begin{proof}
\begin{ithm}

\item Let $p_k$ be the matrix of the permutation $(1,2,\dots, k)$. Let $E={\rm diag}(p_k,p_k)$ be a block diagonal matrix of order $2k$ and let $F={\rm diag}(I_4,E,E,E,E)$ be a block diagonal matrix  of order $8k+4$. A direct calculation shows that $\sigma_1=(F,F)$ is an automorphism of $H$ describing a cycle of order $k$. 

\item Let $q_k$ be the matrix of the permutation $(2,k)(3,k-1)\dots((k+1)/2),(k+3)/2)$, then, define $K={\rm bdiag}(q_k,q_k)$, and $L={\rm diag}(I_4,K,K,K,K)$ of order $8k+4$. A direct calculation shows that the involution $\sigma_2=(L,L)$ is an automorphism of $H$ that swaps $W_i$ and $W_i^\intercal$ within each $W\in \{A,B,C,D\}$. 
\end{ithm}
It remains to prove that $\langle \sigma_1,\sigma_2\rangle$ generate a group isomorphic to $D_{2k}$. We show that $\sigma_1,\sigma_2$ satisfy the presentation $\langle r, s \mid r^{k} = 1,\; s^{2} = 1,\; srs = r^{-1} \rangle$ for $D_{2k}$. It is sufficient to show that $\sigma_2^{-1}\sigma_1\sigma_2 = \sigma_1^{-1}$, which is equivalent to showing $K^{-1}EK = E^{-1}$. A direct calculation yields the result.
\end{proof}

We now identify additional automorphisms of Kimura Hadamard matrices arising from the natural action of the quaternion group on their blocks.

\begin{proposition}\label{automorphisms2}
Let $H$ be a KHM of order $8k+4$, with block matrices $A,B,C,D$, let $i\in \{1,2\}$ and let $W \in \{A,B,C,D\}$. Then the following are automorphisms of $H$.
\begin{ithm}

\item Conjugation by the negative identity matrix, which acts trivially on the rows and columns of the matrix. 

\item Elements $\sigma_3 = (R_1, S_1)$, $\sigma_4 = (R_2,S_2)$ of order 4, where:

\begin{align*}
R_1&={\rm diag}({\rm bdiag}(-1,1,-1,1), ({\rm diag}(-1,1,-1,1)P_1)\otimes I_{2k} ),\\
S_1&={\rm diag}( {\rm diag}(1,-1,-1,1)P_2, ({\rm diag}(-1,1,1,-1)P_1)\otimes I_{2k} ),\\
R_2&={\rm diag}( {\rm diag}(-1,1,1,-1)P_1, ({\rm diag}(1,-1,-1,1)P_2)\otimes I_{2k} ),\\
S_2&={\rm diag}( {\rm bdiag}(1,1,-1,-1), ({\rm diag}(1,1,-1,-1)P_2)\otimes I_{2k} ),\\
\end{align*}
 such that $P_1, P_2$ are the permutation matrices associated to the permutations $(1,2)(3,4),(1,3)(2,4)$ of $S_4$, respectively.
 
\end{ithm}
Moreover, $\sigma_3, \sigma_4$ generate a subgroup of $\mathrm{Aut}(H)$ isomorphic to the quaternion group $Q_8$.
\end{proposition}

\begin{proof}
\begin{ithm}

\item The tuple $(-I_{4(2k+1)},-I_{4(2k+1)})$ is an involution that acts trivially on the rows and columns of $H$.

\item A direct calculation shows that $\sigma_3$ and $\sigma_4$ are automorphisms of $H$ of order 4. Recall that a presentation for the quaternion group $Q_8$ is given by 
$\langle i, j \mid i^{4} = 1,\; i^{2} = j^{2},\; j^{-1} i j = i^{-1} \rangle.$
We have $\sigma_3^2 = \sigma_4^2 = (-I_{4(2k+1)}, -I_{4(2k+1)})$
and  $\sigma_3^{-1}\sigma_4\sigma_3 = \sigma_4 (-I_{4(2k+1)}, -I_{4(2k+1)}) = \sigma_4^3 = \sigma_4^{-1}$.
Hence $\sigma_3,\sigma_4$ generate a group isomorphic to $Q_8$.
\qedhere
\end{ithm}
\end{proof}

For a $y$--invariant KHM, we discovered a further involutory automorphism.

\begin{proposition}\label{y_inv_automorphisms}
Let $H$ be a KHM of order $8k+4$, with $y$--invariant block matrices $A,B,C,D$. Then $H$ has an involutory automorphism that swaps, within each block, the top and bottom halves as well as the left and right halves. 
\end{proposition}
\begin{proof}
Let $P={\rm bdiag}(I_k,I_k)$ and let $Q={\rm diag}(I_4,P,P,P,P)$. Then $\sigma_5=(Q,Q)$ is an automorphism of $H$ that swaps, within each block, the top and bottom halves as well as the left and right halves.
\end{proof}

\begin{corollary}
 \label{coro1}
If $H$ is a KHM of dimension $8k+4$ with $k\geq 3$, then $ (Q_8\times D_{2k})\leq \mathrm{Aut}(H)$. In particular, $|\mathrm{Aut}(H)|\geq 2^4k$. Moreover, if $H$ is $y$--invariant, then $(Q_8\times D_{2k} \times C_2)\leqslant  \mathrm{Aut}(H)$. In particular, $|\text{Aut}(H)|\geqslant 2^5 k$. 
\end{corollary}
\begin{proof}
A direct calculation verifies that the groups  $H_1=\langle \sigma_1, \sigma_2 \rangle \cong D_{2k}$, $H_2 =\langle \sigma_3, \sigma_4 \rangle\cong Q_8 $ and $H_3 = \langle \sigma_5 \rangle\cong C_2$ pairwise commute, and in particular, so do their elements. 
Neither $H_2$ nor $H_3$ intersect $H_1$ non--trivially because $|Z(D_{2k})|=1$. 
Now $|Z(H_2)|=2$, and is generated by $\sigma_3^2 = \sigma_4^2 = (-I_{4(2k+1)}, -I_{4(2k+1)})$ of order 2. But $\sigma_5$ does not lie in this subgroup, hence $H_2$ and $H_3$ intersect trivially. Hence $\langle H_2, H_3 \rangle = H_2\times H_3$. Moreover, every element of $H_2\times H_3$ commutes with every element of $H_1$, but $Z(H_1)=1$. Hence $H_1$ and $H_2\times H_3$ intersect trivially. 
Therefore $\langle H_1,H_2,H_3 \rangle = H_1\times H_2 \times H_3 \cong  D_{2k} \times Q_{8}\times C_2$ as required.
\end{proof}

\subsection{Automorphisms of KHMs arising from $\mathrm{Hol}(D_{2k})$}\label{Hol}

Motivated by the structure of the homomorphisms of $D_{2k}$, in this section, we investigate automorphisms of KHMs that arise as induced actions of automorphisms of $D_{2k}$. We recall some properties of the automorphism group of $D_{2k}$. For $k\geq 3$ it is known that $\text{Aut}(D_{2k})=\{\sigma_{ij}\mid 0\leqslant i,j\leqslant k-1\ \text{and}\ \gcd(i,k)=1\}$, where $\sigma_{ij}:D_{2k}\to D_{2k}$ is defined by $\sigma_{ij}(x)=x^i$ and $\sigma_{ij}(y)=x^jy$. 
It follows that $\text{Aut}(D_{2k})\cong C_k\rtimes C_k^\times$. The \textit{holomorph} of $D_{2k}$ is defined to be $\mathrm{Hol}(D_{2k})=D_{2k} \rtimes \mathrm{Aut}(D_{2k})$.
This group acts on $g \in D_{2k}$ via ${}^{{(h,\phi)}}g = hg^\phi$.

\begin{proposition}\label{holomorph}
Let $H$ be a KHM of order $8k+4$, with block matrices $A,B,C,D$ and corresponding elements $a,b,c,d \in \mathbb{Z}D_{2k}$. Let $\phi\in \text{Aut}(\mathbb{Z}D_{2k})$, with $\mathbb{Z}D_{2k}$ considered as an abelian group. Then
\begin{ithm}
    \item \label{aut2} if $\phi$ fixes the summands of $a,b,c,d$ pointwise then it induces an automorphism of $H$.
    
    \item \label{aut3} if $\phi$ fixes $a,b,c,d$, but not necessarily each of the summands, 
    then it induces an automorphism of $H$.

    \item  \label{aut4} if $\phi$ fixes $a$ and induces a 3-cycle on $\{b,c,d\}$,  then it induces an automorphism of $H$.

\end{ithm}
\end{proposition}

\begin{proof}
\begin{ithm}

\item[\quad a) and b)] Let $\phi \in \operatorname{Aut}(D_{2k})$ and let $P_\phi$ be the permutation matrix induced by $\phi$. Recall that $\rho$ denotes the right regular representation of $D_{2k}$. Thus, the action of $\phi$ on an element $\rho(g)$ is via conjugation, that is $P_\phi^{-1} \rho(g) P_\phi$ for all $g\in D_{2k}$. We claim that $E_\phi = {\rm diag}(I_4, P_\phi^{-1}, P_\phi^{-1}, P_\phi^{-1}, P_\phi^{-1})$ defines an automorphism $ (E_\phi, E_\phi)$ of $H$. To see this, note that for each $X \in \{A, B, C, D\}$, the corresponding element $x=x_1+\dots+x_r\in \mathbb{Z}D_{2k}$ is invariant under $\phi$ by assumption. This implies that $P_\phi X P_\phi^\intercal = X$, which shows the claim.

\item[c)] Let $\phi$ be such an automorphism with corresponding permutation matrix $P_\phi$. Let be $\tau \in S_4$ be the permutation on $\{a,b,c,d\}$ induced by $\phi$.  So $\tau$ fixes $a$ and induces, without loss of generality, the 3-cycle $(b,c,d)$. 
Let $P_1$ and $P_2$ be the permutation matrices of the permutations $(2,3,4)$ and $(1,3,2)$, respectively, and let $P_3={\rm diag}(1,1,-1,-1)P_1$. Let $\sigma_6=(R,S)$ with $R={\rm diag}(P_3,P_3\otimes P_\phi^{-1})$ and $S={\rm diag}(P_2,P_1\otimes P_\phi^{-1})$; a direct calculation shows that $RHS^\intercal = H$. \qedhere

\end{ithm} 
\end{proof}

If the structure of $H$ satisfies certain nice conditions, there are other automorphisms that can arise from the holomorph of $D_{2k}$. The following is an example.

\begin{lemma}\label{aut5} 
Suppose that $H$ is a KHM of order $8k+4$ with block matrices $A$,$B$,$C$,$D$ such that the corresponding elements $a,b,c,d \in \mathbb{Z}D_{2k}$ are y-invariant, and $cy+d= \sum_{g\in D_{2k}}g$. If there exists an automorphism $\phi \in \mathrm{Aut}(D_{2k})$ that 
fixes $a,b$ and swaps $c,d$, then $\phi$ induces an automorphism of $H$.
 \end{lemma}
\begin{proof}
For conciseness, write $H=H(A,B,C,D)$.
Let $P_{\phi}$ be the $2k\times 2k$ induced permutation matrix of $\phi$, and let $M_\phi = {\rm{diag}}(I_4,P_{\phi}^{-1},P_{\phi}^{-1},P_{\phi}^{-1},P_{\phi}^{-1})$. 
Now let $P_1$ be the permutation matrix corresponding to $(1,2)(3,4)$, let $P_2 = {\rm{diag}}(1,1,-1,-1)$ and let $M= {\rm{diag}}(P_1P_2,  I_{2k},I_{2k},-\bar{I}_{2k},-\bar{I}_{2k})$ where $\bar{I}_{2k} = {\rm bdiag}(I_{k},I_{k})$. We calculate that $MHM^\intercal = H(A,B,D,C)$, so $M_{\phi}(MHM^\intercal )M_{\phi}^\intercal = H(A,B,C,D)$. Hence $(R,S)=(M_{\phi}M ,M_{\phi}M)$ induces a strong automorphism of $H$.
 \qedhere
\end{proof}

\begin{remark}
The choice of $M$ in the proof of Lemma \ref{aut5} is not unique. For instance, one can replace $(M,M)$ with $(M_1,M_2)$ where $$M_1= {\rm diag}({\rm bdiag}(1,1,1,1),{\rm bdiag}(I_{2k},I_{2k},\bar{I}_{2k},\bar{I}_{2k}))$$ and 
$$M_2={\rm diag}({\rm diag}(1,-1,-1,1),{\rm bdiag}(-I_{2k},I_{2k},\bar{I}_{2k},-\bar{I}_{2k})).$$\end{remark}

In Section \ref{AutOfKnownEx}, we confirm there are KHMs with automorphisms as described in Proposition \ref{holomorph} and Lemma \ref{aut5}. 

\subsection{Induced action of $\text{Aut}(H)$ on the rows of $H$}
In Section~\ref{B&N} we introduced the permutation group
$\mathcal{A}(H)=\pi(\pi_{1}(\operatorname{Aut}(H)))$, 
where $\pi_{1} : \text{Aut}(H)\to \text{Mon}_{n}(\{\pm 1\})$ is defined by $\pi_{1}(R,S)=R$, and 
$\pi : \pi_{1}(\text{Aut}(H))\to \text{Perm}_{n}$ is defined by $\pi(R)=P_{R}$, where $P_{R}$ denotes the permutation part of the monomial matrix $R=P_{R}D_{R}$. The group $\mathcal{A}(H)$ encodes the permutation action of $\text{Aut}(H)$ on the rows of $H$.
If $H$ is a KHM of order $8k+4$ whose automorphism group has order $2^{4}k$, or $2^{5}k$ and is $y$-invariant, (and hence all automorphisms arise from those described earlier in the section), then Propositions~\ref{automorphisms1} and \ref{automorphisms2} imply that the action of $\mathcal{A}(H)$ on the rows of $H$ induces two orbits,
\[
\mathcal{O}_{1}=\{1,\dots,4\}
\quad\text{and}\quad
\mathcal{O}_{2}=\{5,\dots,8k+4\},
\]
of sizes $4$ and $8k$, respectively.  
The following lemma describes the induced action on these orbits.
\begin{lemma}\label{permgroup}
Let $H$ be a KHM of order $8k+4$ such that $|\text{Aut}(H)|= 2^rk$ with $r\in \{4,5\}$. Further suppose that $H$ is $y$-invariant if $r=5$. Then the following holds. 
\begin{ithm}
\item  The induced action of $\mathcal{A}(H)$ on $\mathcal{O}_1$ is imprimitive with system of imprimitivity \[\{\{1,3\},\{2,4\}\}.\] 
Moreover, the induced action on blocks is isomorphic to $C_2^2$.
\item The induced action of $\mathcal{A}(H)$ on $\mathcal{O}_2$ is imprimitive with system of imprimitivity 
\[\{\{5,\dots,k+4\},\dots, \{7k+5,\dots, 8k+4\}\}.\]
The induced action on blocks is isomorphic to $C_2^3$.  If $r=4$ then the induced action on each block is isomorphic to $C_k$ and if $r=5$ then the induced action on each block is isomorphic to $D_{2k}$. 
\end{ithm}
\end{lemma}
\begin{proof}
\begin{ithm}
\item The automorphisms $\sigma_3, \sigma_4$ in Proposition \ref{automorphisms2} induce the permutations $(1,4)(3,2), (1,2)(3,4)$ which respect the partition $\{\{1,3\},\{2,4\}\}$, whereas the automorphisms $\sigma_1,\sigma_2, \sigma_5$ induce the identity element. It follows that the induced action of $\mathcal{A}(H)$ on the blocks $\{\{1,3\},\{2,4\}\}$ is isomorphic to $C_2^2$, and the induced action on each block is isomorphic to $C_2$. 
\item Let $B_1=\{5,\dots,k+4\},\dots, B_8=\{7k+5,\dots, 8k+4\}$. The automorphisms  $\sigma_2,\sigma_3,\sigma_4$ in Propositions \ref{automorphisms1} and \ref{automorphisms2} induce the permutations 
\[(1,3)(2,4)(5,7)(6,8), (1,5)(2,6)(3,7)(4,8), (1,2)(3,4)(5,6)(7,8),\]  respectively, acting on the indices of the partition $B_1,\dots, B_8$; the automorphism $\sigma_1$ induces the identity element. These permutations are involutions, commute with each other, respect the partition $B_1,\dots,B_8$, and generate a group isomorphic to $C_2^3$.  Now, only $\sigma_1$ induces a non-trivial permutation which stabilises $B_1$ setwise and induces a $k$-cycle on $B_1$. It follows that if $r=4$ the induced action on each block is isomorphic to $C_k$. If $r=5$, then $H$ is $y$--invariant. Here, the automorphism $\sigma_2\sigma_5$ derived from Propositions \ref{automorphisms1} and \ref{y_inv_automorphisms} induces an involution that stabilises $B_1$ setwise. Noting that $\sigma_1$ and $\sigma_2\sigma_5$ satisfy a presentation for $D_{2k}$, we deduce that the action on each block is isomorphic to $D_{2k}$. \qedhere
\end{ithm}
\end{proof}
\begin{corollary}
If $k=3$ and there exists an automorphism of $D_{6}$ as described in Proposition \ref{holomorph}(c) then the induced action of $\mathcal{A}(H)$ on $\mathcal{O}_1$ is primitive and isomorphic to $\text{Alt}_4$. Moreover, the induced action of $\mathcal{A}(H)$ on $\mathcal{O}_2$ is imprimitive and the action on the blocks is isomorphic to $C_2\times \text{Alt}_4$. Moreover the induced action on each block is isomorphic to $C_3$.
\end{corollary}
\begin{proof}
If $k=3$ then in addition to the permutations described in the proof of Lemma \ref{permgroup}(a), the group $\mathcal{A}(H)$ also contains the permutation $(2,3,4)$ which does not preserve any 2-subset partition of the set $\{1,2,3,4\}$. This shows that the induced action of $\mathcal{A}(H)$ on $\mathcal{O}_1$ is primitivity and isomorphic to $\text{Alt}_4$. The automorphism $\sigma_6$ from Proposition \ref{holomorph}(c) induces the permutation $(3,7,5)(4,8,6)$ on the blocks $B_1,\dots,B_8$; a quick check shows that the induced action of $\mathcal{A}(H)$ on blocks is isomorphic to $C_2\times \text{Alt}_4$. Finally, the permutations induced by $\sigma_1$ and $\sigma_6$ are the only ones that fix $B_1$ setwise; from this we determine that the induced action on each block is isomorphic to $C_3$. 
\end{proof}

\subsection{Automorphisms of the known examples of KHMs}\label{AutOfKnownEx}
Having constructed a number of automorphisms of KHMs, we now analyse the automorphism groups of the known examples.
In this section, we present in Tables \ref{tab1} and \ref{tab2} the list of automorphism groups of the matrices found by Kimura and Niwasaki, as well as the matrices introduced by Shinoda and Yamada for admissible $p\leqslant 601$. These groups were calculated using the package {\sc{Magma}} \cite{MAGMA}. In his masters thesis, \'O Cath\'ain in \cite[Section 7.4.1]{POCMT} conjectured that the order of the automorphism group of a KHM, of order $8k+4$, is $2^5k$.
The data in Tables \ref{tab1} and \ref{tab2} show that this conjecture is false as they reveal automorphisms with orders $2^4\cdot 11$, $2^4\cdot 3^2$, $2^6\cdot k$ and $2^7k$ for various values $k$. 

On further investigation, we find that all of the automorphism groups of the known KHMs that we computed can be generated by automorphisms described in Sections \ref{general_auts} and \ref{Hol}. Indeed, it is often the case that the full automorphism group of a KHM is isomorphic to the group $C_2\times D_{2k} \times Q_8$ discussed in Corollary \ref{coro1}. We now briefly discuss the other automorphisms that arise from $\mathrm{Hol}(D_{2k})$.

\renewcommand{\arraystretch}{1.2}
\begin{table}[h] 
\begin{tabular}[b]{ cc }   
\begin{tabular}[t]{r|l|l}
$k$  & Automorphism Group & Order \\  \hline 
3 & $(Q_8:C_3)\times D_{2\cdot 3}$  & $2^4\cdot 3^2$\\
5 & $C_2\times Q_8\times D_{2\cdot 5}$ &  $2^5\cdot 5$\\
7 & $C_2\times Q_8\times D_{2\cdot 7}$ & $2^5\cdot 7$\\
9 & $C_2\times Q_8\times D_{2\cdot 9}$ & $2^5\cdot 3^2$\\
11 & $Q_8\times D_{2\cdot 11}$ & $2^4\cdot 11$ \\
13 & $C_{13}:(C_2.(C_4\times Q_8))$ & $2^6\cdot 13$\\
17 & $C_{17}:(C_2.(C_4\times Q_8))$ & $2^6\cdot 17$
\end{tabular}&  
\hspace{2ex}
\begin{tabular}[t]{r|l|l}
$k$  & Automorphism Group & Order\\  \hline 
19 & $C_{19}:(C_2^2\times Q_8)$ & $2^5\cdot 19$\\
21 & $C_{21}:(C_2^2\times Q_8)$ & $2^5\cdot 3\cdot 7$\\
23 & $C_{23}:(C_2^2\times Q_8)$ & $2^5\cdot 23$\\
25 & $C_{25}:(C_2.(C_4\times Q_8))$ & $2^6\cdot 5^2$\\
27 & $C_{27}:(C_2^2\times Q_8)$ & $2^5\cdot 3^3$\\
29 & $C_{29}:(C_2^2\times Q_8)$ & $2^5\cdot 29$\\
41 & $C_{41}:(C_2.(C_4\times Q_8))$ & $2^6\cdot 41$
\end{tabular} \\ \\
\end{tabular}

\caption{Automorphism groups (isomorphism type) and their orders for the KHMs of order $4(2k+1)$ for $k=$ 3-13, 17-29, 41, found by Kimura and Niwaski. Here $G=N:H$ means $N\lhd G$, $H\leqslant G$ and $G/N\cong H$; also $G=N.H$ means $N\lhd G$ and $G/N\cong H$.}\label{tab1}
\end{table}

\renewcommand{\arraystretch}{1.2}
\begin{table}[h] 
\begin{tabular}[b]{ cc }   
\begin{tabular}[t]{r|l|l}
$p$  & Automorphism Group & Order \\  \hline 
5 & $C_2\times Q_8\times D_{2\cdot 5}$  & $2^5\cdot 5$\\
13 & $C_{13}:(C_2.(C_4\times Q_8))$ &  $2^6\cdot 13$\\
37 & $C_2\times Q_8 \times D_{2\cdot 37}$ & $2^5\cdot 37$\\
41 & $C_{41}:(C_2\times C_4\times Q_8)$ & $2^6\cdot 41$\\
61 & $C_2\times Q_8 \times D_{2\cdot 61}$ & $2^5\cdot 61$ \\
97 & $C_2\times Q_8 \times D_{2\cdot 97}$ & $2^5\cdot 97$\\
157 & $C_{157}:(C_2^2\times Q_8)$ & $2^5\cdot 157$
\end{tabular}&  
\hspace{2ex}
\begin{tabular}[t]{r|l|l}
$p$  & Automorphism Group & Order\\  \hline 
181 & $C_2\times Q_8 \times D_{2\cdot 181}$ & $2^5\cdot 181$\\
229 & $C_{229}:(C_2^2\times Q_8)$ & $2^5\cdot 229$\\
313 & $C_{313}:(C_2.(C_8\times Q_8))$ & $2^7\cdot 313$\\
337 & $C_2\times Q_8 \times D_{2\cdot 337}$ & $2^5\cdot 337$\\
421 & $C_{421}:(C_2.(C_4\times Q_8))$ & $2^6\cdot 421$\\
577 & $C_{577}:(C_2^2\times Q_8)$ & $2^5\cdot 477$\\
601 & $C_{601}:(C_2^2\times Q_8)$ & $2^5\cdot 601$
\end{tabular} \\ \\
\end{tabular}

\caption{Automorphism groups (isomorphism type) and their orders for the KHMs of order $4(2p+1)$, with $p\leqslant 601$, $p\equiv 1\bmod 4$ and $2p-1$ a prime power, obtained by the Shinoda and Yamada's construction.}\label{tab2}
\end{table}

The automorphism $(R, S)$ constructed in the proof of Proposition \ref{holomorph}(c) has order $3$. For $k = 3$, the elements $a = 1+x$, $b = 1+x^2+x^2y$, $c = 1+x^2+xy$, and $d = 1+x^2+y$ in $\mathbb{Z}D_6$ satisfy Equations \eqref{eq1} and \eqref{eq12} and yield a KHM of order $28$. The automorphism $\phi \in \mathrm{Aut}(D_{2k})$, defined by $\phi(x)=x$ and $\phi(y) = xy$, fixes the element $a$ and induces a $3$--cycle on $\{b, c, d\}$. By Proposition \ref{holomorph}(c), it follows that there exists an additional element of order $3$, distinct from the $3$--cycle described in Proposition \ref{automorphisms1}. This explains the additional copy of $C_3$ in the automorphism group of the KHM of order $28$, found by Kimura and Niwasaki, which in turn explains the order $2^4 \cdot 3^2$ of such an automorphism group.

For $k = 13$, the elements 
\[
\begin{aligned}
a &= x + x^4 + x^5 + x^6 + x^7 + x^8 + x^9 + x^{12} + (x^2 + x^3 + x^{10} + x^{11})y, \\
b &= 1 + x^4 + x^6 + x^7 + x^9 + (x + x^2 + x^3 + x^5 + x^8 + x^{10} + x^{11} + x^{12})y, \\
c &= 1 + x^2 + x^5 + x^6 + x^7 + x^8 + x^{11} + (x^2 + x^5 + x^6 + x^7 + x^8 + x^{11})y, \\
d &= 1 + x + x^3 + x^4 + x^9 + x^{10} + x^{12} + (x + x^3 + x^4 + x^9 + x^{10} + x^{12})y,
\end{aligned}
\]
in $\mathbb{Z}D_{13}$ satisfy Equations \eqref{eq1} and \eqref{eq12} and yield a KHM $H$ of order $108$. These elements also satisfy the conditions of Proposition \ref{aut5}, with automorphism $\phi$ defined by $\phi(x) = x^5$ and $\phi(y) = y$. Consequently, $\phi$ induces an automorphism of $H$ of order $4$, which is not contained in the subgroup of $\text{Aut}(H)$ isomorphic to  $C_2\times D_{26}\times Q_{8}$. This explains the additional copy of $C_4$ in the automorphism group of the KHM of order $108$, found by Kimura and Niwasaki. This also explains the order $2^6 \cdot 13$ of this automorphism group.

Finally, for $k = 313$, the construction by Shinoda and Yamada (see Section \ref{KMSExp}) yields elements $a, b, c$, and $d$ satisfying Equations \eqref{eq1} and \eqref{eq12} , which produces a KHM of order $2508$. These elements also meet the conditions of Proposition \ref{aut5}, with automorphism $\phi$ defined by $\phi(x) = x^5$ and $\phi(y) = y$. In particular, $\phi$ induces an automorphism of order $8$. This explains the copy of $C_8$ in the automorphism group of the KHM of order $2508$, discovered by Shinoda and Yamada, and its automorphism group order $2^7 \cdot 313$.


\subsection{Further properties of automorphisms of KHMs}
In this final subsection, we prove some further properties of automorphisms of KHMs. First we introduce a definition. A \textit{block monomial matrix} is a block matrix in which each block row and each block column contains exactly one nonzero block, and all other blocks are zero, and each nonzero block is a $\{\pm 1\}$-monomial matrix.

\begin{proposition}\label{trivialaut}
Let $H$ be a KHM of order $n=8k+4$ and let $\sigma=(R,S)\in \text{Aut}(H)$. 
\begin{ithm}
\item If $\sigma$ is a strong automorphism of $H$ that preserves the first four rows of $H$ and permutes the block matrices $A,B,C,D$ of $H$, then it is trivial.
\item If $R=\text{diag}(r_1,\dots,r_n)$ is a diagonal matrix, and $S$ fixes one column of $H$ then $\sigma$ is trivial.
\item If $R$ is diagonal and $r_1=-1$ then $\sigma=(-I,-I)$.
\item If $R$ and $S$ are diagonal then $\sigma\in \{(I,I),(-I,-I)\}$.
\end{ithm}
\end{proposition}

\begin{proof}
\begin{ithm}

\item Let $\sigma=(R,R)$ be a strong automorphism that permutes the blocks of $H$, and let $E_1$ and $E_2$ be the sub-matrices of $H$ obtained from its first 4 rows and columns, respectively. Since $\sigma$ preserves the blocks of $H$, the matrix $R$ has the form ${\rm diag}(R_1,R_2)$, where $R_1$ is a $4\times 4$ monomial matrix and $R_2$ is an $8k\times 8k$ block monomial  matrix. Observe that $\sigma$ can preserve the first four rows and columns of $H$ only if $R_1 = I_4$. If $R_2$ is non-trivial, then it permutes and/or multiplies by $-1$ some of the last $8k$ rows and the last $8k$ columns of $H$; this induces permutations and multiplication by $-1$ to some of the columns in the matrix $E_1$, and some of the rows in $E_2$.  However, this is not possible as $R=I_4$ cannot counteract the changes induced by $R_2$ in $E_1$ and $E_2$. This forces that $R_2$ is also trivial, which shows the statement. 


\item Assume that the column $j$ of $H=(h_{ij})$ is fixed by $Q^\intercal$. It follows from the equation $DHQ^\intercal = H$ that $(r_1 h_{1,j},\dots,r_n h_{n,j})^\intercal = (h_{1,j},\dots,h_{n,j})^\intercal$. This shows that $r_1=\dots=r_n=1$, and the result follows. 

\item Consider the equation $RH=HS$. The effect of multiplying $H$ by $R$ from the left shows that the first row of $H$ is multiplied by $-1$. Since $S$ permutes the columns of $H$ and/or multiply them by $-1$, it follows that $S=-P$ for a permutation matrix $P$. It follows that every row of $H$ is multiplied by $-1$; which shows that $r_i=-1$ for all $i$. This shows that $R=S=-I$. \qedhere

\item This follows from b) and c). 

\end{ithm}
\end{proof}

The lemma below provides information about automorphisms of KHMs that stabilise the first row and the first four rows of the matrix.

\begin{proposition}
Let $H$ be a KHM, and $(R,S)$ be an automorphism of $H$. 
\begin{ithm}
    \item If $R$ preserves the first row of $H$, then $S$ is a permutation matrix.
    \item If $R$ stabilises the first four rows of $H$, then $R$ and $S$ preserve the first four rows and block structure of $H$. In particular, each block of $H$ is fixed by $(R,S)$.
\end{ithm}
\end{proposition}
\begin{proof}
For \rm{a}), recall the form of a KHM from Table \ref{TCCKimura}.
If $R$ fixes the first row of $H$, which consists of only 1's, then $S$ cannot negate any columns, and hence is a permutation matrix. For \rm{b}), suppose that $R$ stabilises the first four rows of $H$. As $S$ is a permutation matrix, two columns $c_a,c_b$ in the same orbit under $S$ must have $h_{i,a} = h_{i,b}$ for each $i \in \{1,2,3,4\}$. It follows that $S$ fixes each of the first 4 columns and the block structure of $H$. Since $S$ fixes the first four columns, two rows $r_c$, $r_d$ in the same orbit under $R$ must have $h_{c,j} = h_{d,j}$ for each $j \in \{1,2,3,4\}$. Therefore, $R$ preserves the block structure of $H$. In particular, the blocks of $H$ are fixed under $(R,S)$.
\end{proof}


\section{Open questions}\label{Discussion}

All automorphisms \((R,S)\) of KHMs described in Section~\ref{AutsKHMs} have the block–diagonal form
\[
R = \mathrm{diag}(R_{1},R_{2}), \qquad 
S = \mathrm{diag}(S_{1},S_{2}),
\]
where \(R_{1},S_{1} \in \mathrm{Mon}_{4}(\{\pm 1\})\) and \(R_{2},S_{2}\) are $8k\times 8k$ block monomial matrices, with blocks of size $2k\times 2k$. In turn, each $2k\times 2k$--block is a block--monomial matrix, with sub-blocks of size $k\times k$, see for example Propositions \ref{automorphisms1}, \ref{automorphisms2} and \ref{holomorph}. The automorphism groups of the matrices constructed by Kimura and Niwasaki, and by Shinoda and Yamada (for \(p \leqslant 601\)), consist entirely of automorphisms of the same form with the same constraint on \(R_{1},S_{1}\) and \(R_{2},S_{2}\). 
In light of this consistent structure across all known constructions and examples, we conjecture the following.

\begin{conjecture}
Every automorphism $(R,S)$ of a KHM has the aforementioned block monomial form. 
\end{conjecture}

For $H$ a HM of order $4n$; Tonchev \cite[Theorem 2.1]{Tonchev} showed that if $p$ is a prime dividing $|\text{Aut}(H)|$ then one of the following occurs: $p$ divides $4n$ or $4n-1$ or $p\leqslant 2n-1$. If we assume that every automorphism of a KHM has the block monomial form described above, and $p$ divides  $|\text{Aut}(H)|$ then $p\leqslant k$. Moreover, if $p>3$, then an automorphism $(R,S)$ of order $p$ has to preserve the first four rows (and columns) of $H$. This shows that $R=\text{diag}(I_4,R')$ and $S=\text{diag}(I_4,S')$. Now $(R',S')$ can only act within the blocks $A,B,C,D$ of $H$, as otherwise it would change the signs of the first four rows (and columns) of $H$, which $I_4$ cannot counteract. The same argument shows that $(R',S')$ does not move the blocks $A,B,C,D$ of $H$ from their positions. Thus, we have the following.  

\begin{lemma}
Let $H$ be a KHM of order $8k+4$, let $p>3$ be a prime dividing $|\text{Aut}(H)|$ and let $(R,S)$ be a block monomial automorphism of order $p$. Then 
\[R={\rm diag}(I_4,R_{1},\dots,R_{8})\ \text{ and }\ S={\rm diag}(I_4,S_{1},\dots,S_{8})\] 
where $R_{i},S_{i}\in \text{Mon}_{k}(\pm)$ with $i\in \{1,\dots,8\}$.
\end{lemma}

If the automorphism $(R,S)$ has order $p>3$, then it follows from the equation $RH=HS$ that 
\[
R_{1} [\begin{array}{cccc}\bold{1}^\intercal & \bold{1}^\intercal & \bold{1}^\intercal & -\bold{1}^\intercal \end{array}]= [\begin{array}{cccc}\bold{1}^\intercal & \bold{1}^\intercal & \bold{1}^\intercal & -\bold{1}^\intercal \end{array}].
\]
This shows that $R_{1}$ is a permutation matrix. A similar argument shows that $R_{i}$ and $S_{i}$ are permutation matrices for $i=1,\dots,8$. 

It follows that the action of an automorphism $(R,S)$ of order $p$, where $3<p\leqslant k$ is prime, on a KHM $H$ of order $8k+4$ is completely determined by how the components $R_{i}$ and $S_{i}$ act on the blocks $W_1,W_2,W_1^\intercal,W_2^\intercal$ of each $W\in \{A,B,C,D\}$. All the examples presented in this work suggest the following.

\begin{conjecture}
Let $H$ be a KHM of order $4(2k)+4$. If $(R,S) \in \mathrm{Aut}(H)$ has prime order $p$, then $p=2$ or $p$ divides $k$.
\end{conjecture}

Finally, we turn out attention back to Section \ref{AutOfKnownEx}; in view of the data compiled in Tables \ref{tab1} and \ref{tab2}, and the discussion throughout such section,
we conjecture the following. 

\begin{conjecture}
The order of the automorphism group of a KHM $H$ satisfies $|\text{Aut}(H)|=2^r\cdot 3^s \cdot k$, where $r\geqslant 4$ and $s\in \{0,1\}$.    
\end{conjecture}



\end{document}